\documentclass[12pt]{amsart}
\usepackage{amssymb}
\usepackage{amsfonts}

\theoremstyle{plain}

\newtheorem{definition}{Definition}
\newtheorem{example}{Example}

\newtheorem{remark}{Remark}

\newtheorem{theorem}{Theorem}
\numberwithin{equation}{section}

\input{tcilatex}

\begin{document}
\title{Measures of localization and quantitative Nyquist densities}
\author{Lu\'{\i}s Daniel Abreu}
\address{Austrian Academy of Sciences, Acoustics Research Institute,
Wohllebengasse 12-14, Vienna A-1040, Austria. }
\email{labreu@kfs.oeaw.ac.at}
\author{Jo\~{a}o M. Pereira}
\address{Program in Applied and Computational Mathematics, Princeton
University, NJ 08544, USA}
\email{jpereira@princeton.edu}
\thanks{}
\thanks{L.D.~Abreu was supported by COMPETE/FEDER via CMUC/FCT project
PTDC/MAT/114394/2009  and by Austrian Science Foundation (FWF) START-project
FLAME ("Frames and Linear Operators for Acoustical Modeling and Parameter
Estimation", Y 551-N13)\\
Jo\~{a}o M. Pereira was supported by a Caloust Gulbenkian Foundation grant
\textquotedblleft Novos talentos em Matem\'{a}tica\textquotedblright .}

\begin{abstract}
We obtain a refinement of the degrees of freedom estimate of Landau and
Pollak. More precisely, we estimate, in terms of $\epsilon $, the increase
in the degrees of freedom resulting upon allowing the functions to contain a
certain prescribed amount of energy $\epsilon $ outside a region delimited
by a set $T$ in time and a set $\Omega $ in frequency. In this situation,
the lower asymptotic Nyquist density $\left\vert T\right\vert \left\vert
\Omega \right\vert /2\pi $ is increased to $(1+\epsilon )\left\vert
T\right\vert \left\vert \Omega \right\vert /2\pi $. At the technical level,
we prove a \emph{pseudospectra} version of the classical spectral dimension
result of Landau and Pollak, in the multivariate setting of Landau.
Analogous results are obtained for Gabor localization operators in a compact
region of the time-frequency plane.
\end{abstract}

\maketitle
\dedicatory{\emph{\textquotedblleft It is easy to argue that real signals
must be bandlimited. It is also easy to argue that they cannot be
so\textquotedblright , }David Slepian, \emph{On Bandwith}, 1976.}

\section{Introduction}

\subsection{The Nyquist rate and Landau-Pollack degrees of freedom estimate}

Let $D_{T}$\ and $B_{\Omega }$\ denote the operators which cut the time
content outside $T$ and the\ frequency content outside $\Omega $,
respectively.\ In the fundamental paper \cite{LP}, whose purpose was to 
\emph{examine the true in the engineering intuition that there are
approximately }$\left\vert T\right\vert \left\vert \Omega \right\vert /2\pi $%
\emph{\ independent signals of bandwidth }$\Omega $\emph{\ concentrated on
an interval of length }$T$, Landau and Pollak have considered the eigenvalue
problem associated with the positive self-adjoint operator%
\begin{equation}
P_{T,\Omega }=D_{T}B_{\Omega }D_{T}  \label{eigen}
\end{equation}%
When $T$ and $\Omega $ are real intervals, the operator\ involved in this
problem can be written explicitly as%
\begin{equation*}
(P_{T,\Omega }f)(x)=\left\{ 
\begin{array}{c}
\int_{T}\frac{\sin \Omega (x-t)}{\pi (x-t)}f(t)dt\text{ \ \ \ }if\text{ \ }%
x\in T \\ 
0\text{ \ \ \ \ \ \ \ \ \ \ \ \ \ \ \ \ \ \ \ \ \ \ \ }if\text{ \ }x\notin T%
\end{array}%
\right. .
\end{equation*}%
The cornerstone of the results in \cite{LP} is the following asymptotic
estimate for the number of eigenvalues $\lambda _{n}$ of (\ref{eigen}) which
are close to one: 
\begin{equation}
\#\{n:\lambda _{n}>1-\delta \}\simeq \left\vert T\right\vert \left\vert
\Omega \right\vert /2\pi +C_{\delta }\log \left( \left\vert T\right\vert
\left\vert \Omega \right\vert \right) \text{,}  \label{estimate}
\end{equation}%
as $T\rightarrow \infty $, where $C_{\delta }$ is a constant depending only
on $\delta $. Since the eigenvalues of the operator (\ref{eigen}) are the
same as those of the operator $B_{\Omega }D_{T}$, whose eigenfunctions $f$
satisfy 
\begin{equation*}
\int_{_{T}}\left\vert f\right\vert ^{2}=\lambda \left\Vert f\right\Vert ^{2}%
\text{,}
\end{equation*}%
the estimate (\ref{estimate})\ provides us with the number of orthogonal
eigenfunctions $f$ of (\ref{eigen}), such that%
\begin{equation*}
\int_{_{T}}\left\vert f\right\vert ^{2}\sim \left\Vert f\right\Vert ^{2}%
\text{,}
\end{equation*}%
asymptotically when $T\rightarrow \infty $.\ Within mathematical signal
analysis (see, for instance the discussion in \cite[pg. 23]{Dau} and the
recent book \cite{HogLak}), (\ref{estimate}) is viewed as a mathematical
formulation of the Nyquist rate, the fact that a time- and bandlimited
region $T\times \Omega $ corresponds to $\left\vert T\right\vert \left\vert
\Omega \right\vert /2\pi $ degrees of freedom. In other words, there exist,
up to a small error, $\left\vert T\right\vert \left\vert \Omega \right\vert
/2\pi $ independent functions that are essentially timelimited to $T$ and
bandlimited to $\Omega $.

The main goal of this note is to refine the degrees of freedom estimate (\ref%
{estimate}) in the context to be made precise in the next subsection.

\subsection{A refinement of Landau-Pollack degrees of freedom estimate}

Ideally, one would like to count the number of orthogonal functions in $%
L^{2}(\mathbb{R})$, which are time and band-limited to a bounded region like 
$T\times \Omega $. Unfortunately, such functions do not exist (because
band-limited functions are analytic). As a result, it is natural to count
the number of orthonormal functions in $L^{2}(\mathbb{R})$, which are
approximately time and band-limited to a bounded region like $T\times \Omega 
$. An optimal solution to this problem is given by the number of
eigenfunctions of (\ref{eigen}) whose eigenvalues are very close to one in
the sense that they exceed a threshold $1-\delta $, leading to (\ref%
{estimate}). We remark that estimate (\ref{estimate}) counts the degrees of
freedom in spaces generated by the so-called prolate spheroidal wave
functions (see \cite{prolates} for a recent reference on these functions).
Our count of degrees of freedom will be based on different functions (but
the new functions will be constructed using the prolate spheroidal
functions).

Our purpose is to refine (\ref{estimate}), by taking advantage of the fact
that most of the eigenvalues of $P_{T,\Omega }$ are closer to $1$ than to $%
1-\delta $. To get an estimate of the space of functions satisfying simply $%
\left\Vert P_{T,\Omega }f-f\right\Vert \geq $ $1-\delta $, we can replace $n$
orthogonal eigenfunctions (which, in the case of the interval described in
this introduction are the prolate spheroidal wave functions \cite{prolates}\
)\ of $P_{T,\Omega }$ whose eigenfunctions are close to $1$, with $n+1$
orthogonal functions with $\left\Vert P_{T,\Omega }f-f\right\Vert \approx $ $%
1-\delta $. Essentially, we split the well concentrated energy of the $n$
prolate functions among $n+1$ vectors and add an extra dimension to obtain
an orthogonal set. This idea will allow one to build a set of orthonormal
functions in $L^{2}(\mathbb{R})$, which is a bit less concentrated than the
prolates, so that it contains a prescribed \emph{quantity }$\epsilon $\emph{%
\ of time-frequency content outside the bounded region} $T\times \Omega $.
Precisely, we will count the number of orthogonal functions in $L^{2}(%
\mathbb{R})$, $\epsilon $-\emph{localized} in the sense that%
\begin{equation}
\left\Vert P_{T,\Omega }f-f\right\Vert ^{2}\leq \epsilon \text{.}  \label{a}
\end{equation}%
From our main result it follows that (\ref{estimate}) has the following
analogue in this setting: if $\eta (\epsilon ,T,\Omega )$ stands for the
maximum number of orthogonal functions of $L^{2}({\mathbb{R}})$ satisfying (%
\ref{a}), then, as $\left\vert T\right\vert \rightarrow \infty $,%
\begin{equation}
\frac{\left\vert T\right\vert \left\vert \Omega \right\vert }{2\pi }%
(1+\epsilon )+C_{\delta }\log \left( \left\vert T\right\vert \left\vert
\Omega \right\vert \right) \leq \eta (\epsilon ,T,\Omega )\leq \frac{%
\left\vert T\right\vert \left\vert \Omega \right\vert }{2\pi }(1-2\epsilon
)^{-1}+C_{\delta }\log \left( \left\vert T\right\vert \left\vert \Omega
\right\vert \right) \text{.}  \label{estimateintrol2}
\end{equation}

\subsection{Localization operators}

Our understanding of the concentration problem is based on the study of
operators which localize signals in bounded regions of the time-frequency
plane. Such operators are known in a broad sense as time-frequency
localization operators; their eigenfunctions are orthogonal sequences of
functions with optimal concentration properties. The quantitative
formulation of the concentration problem can be seen in terms of
localization operators as follows: rather than looking for the optimal
concentrated functions in a given region of the time-frequency plane, we
will allow the functions to contain a certain prescribed amount of energy
outside the given region, and estimate the resulting increase in the degrees
of freedom. Given an operator $L$, instead of counting the eigenfunctions of%
\begin{equation*}
Lf=\lambda f
\end{equation*}%
associated with eigenvalues $\lambda $ close to one, we will count
orthogonal functions $\epsilon $-\emph{localized} with respect to $L$ in the
sense that 
\begin{equation}
\left\Vert Lf-f\right\Vert ^{2}\leq \epsilon \text{.}  \label{b}
\end{equation}%
In the next paragraph we will see how the idea of $\epsilon $-localization
arises from the concept of pseudospectra of linear operators.

\subsection{Pseudospectra and $\protect\epsilon -$localization}

The result of Landau and Pollak has later been improved by Landau to several
dimensions and more general sets than intervals in \cite{Landau1967} and 
\cite{Landau1975}. Also in \cite{Landau1975}, Landau introduced the concept
of $\epsilon $-approximated eigenvalues and eigenfunctions. This concept is
a forerunner of what is nowadays known as the \emph{pseudospectra} in the
numerical analysis of non-normal matrices \cite{TE}. Recent developments in
spectral approximation theory involve the concept of $n$-pseudospectrum,
which has been introduced in \cite{Hansen2} with the purpose of
approximating the spectrum of bounded linear operators on an infinite
dimensional, separable Hilbert space, and then used in the proof of the
computability of the spectrum of a linear operator on a separable Hilbert
space \cite{Hansen}. We will recall Landau's original definition, which was
the following:

\begin{definition}
$\lambda $ is an $\epsilon $-approximated eigenvalue of $L$ if there exists $%
f$ with $\left\Vert f\right\Vert =1$, such that $\left\Vert Lf-\lambda
f\right\Vert \leq \epsilon $. We call $f$ an $\epsilon $-approximated
eigenfunction corresponding to $\lambda $.
\end{definition}

Thus, our quantitative measure (\ref{b}) for\ the time-frequency
localization of $f$ is equivalent to $f$\ being a $\epsilon $-approximated
eigenfunction corresponding to $1$.

\begin{example}
Suppose that $\varphi $ is an eigenfunction of $P_{rT,\Omega }$ with
eigenvalue $\lambda $. Then%
\begin{equation*}
\left\Vert P_{rT,\Omega }\varphi -\varphi \right\Vert =1-\lambda \text{.}
\end{equation*}%
Thus, every eigenfunction of $P_{rT,\Omega }$ is a $(1-\lambda )$%
-pseudoeigenfunction of $P_{rT,\Omega }$ with pseudoeigenvalue $1$.
\end{example}

The relevant fact is that the number of orthogonal pseudoeigenfunctions with
pseudoeigenvalue greater than a given threshold is larger than the number of
eigenfunctions with eigenvalue greater than that threshold. A large class of
functions satisfying (\ref{estimate}) arises from the set of almost
bandlimited functions in the sense of Donoho-Stark's concept of $\epsilon $%
-concentration.

\begin{example}
According to \cite{DonohoStark}, $f$ is $\epsilon _{T}$-concentrated in $T$
if 
\begin{equation*}
\left\Vert D_{T}f-f\right\Vert \leq \epsilon _{T}\text{ }
\end{equation*}%
and its Fourier transform $Ff$ (see definitions in the next section) is $%
\epsilon _{\Omega }$-concentrated in $\Omega $\ if 
\begin{equation}
\left\Vert B_{\Omega }f-f\right\Vert \leq \epsilon _{\Omega }\text{. }
\label{almostbandlim}
\end{equation}%
An application of the triangle inequality shows that if $f$ is $\epsilon
_{T} $-concentrated in $T$ and $Ff$ is $\epsilon _{\Omega }$-concentrated in 
$\Omega $\ \ then 
\begin{equation*}
\left\Vert B_{\Omega }D_{T}f-f\right\Vert \leq \epsilon _{T}+\epsilon
_{\Omega }\text{. }
\end{equation*}%
and another application of the triangle inequality gives%
\begin{equation}
\left\Vert P_{T,\Omega }f-f\right\Vert \leq 2\epsilon _{T}+\epsilon _{\Omega
}\text{.}  \label{epsilonconc}
\end{equation}%
Thus, if $f$ is $\epsilon _{T}$-concentrated in $T$ and$\ Ff$ is $\epsilon $%
-concentrated in $\Omega $, then $f$ is a $(2\epsilon _{T}+\epsilon _{\Omega
})$-pseudoeigenfunction of $P_{rT,\Omega }$ with pseudoeigenvalue $1$.
\end{example}

One should notice that these notions, as well as the topic investigated in
this note, can be related to Slepian's philosophical and mathematical quest 
\cite{Slepian}, aiming at solving the bandwidth paradox: \emph{%
\textquotedblleft It is easy to argue that real signals must be bandlimited.
It is also easy to argue that they cannot be so\textquotedblright\ }\cite%
{Slepian}

\subsection{Organization of the paper}

This is essentially a single-result paper, which is Theorem 1 in the next
section. We first provide some background concerning Landau's results about
the extension of the time-band limiting problem to functions in $\mathbb{R}%
^{d}$, bandlimited to a set of finite measure and the main notations. The
last section of the paper is devoted to another important class of operators
where our results apply, namely Gabor localization operators. Since the
proofs for Gabor localization operators are very similar to those in section
2, they are omitted.

\section{Notations and main results}

\subsection{Time- and band- limiting operators}

A description of the general set-up of \cite{Landau1967} and \cite%
{Landau1975} follows. The sets $T$ and $\Omega $ are general subsets of
finite measure of $%
\mathbb{R}
^{d}$. Let%
\begin{equation*}
Ff(\xi )=\frac{1}{(2\pi )^{d/2}}\int_{\mathbb{R}^{d}}f(t)e^{-i\xi t}dt
\end{equation*}%
denote the Fourier transform of a function $f\in L^{1}({\mathbb{R}}^{d})\cap
L^{2}({\mathbb{R}}^{d})$. The subspaces of $L^{2}({\mathbb{R}}^{d})$
consisting, respectively, of the functions supported in $T$ and of those
whose Fourier transform is supported in $\Omega $ are%
\begin{eqnarray*}
\mathcal{D}(T) &=&\{f\in L^{2}({\mathbb{R}}^{d}):f(x)=0,x\notin T\} \\
\mathcal{B}(\Omega ) &=&\{f\in L^{2}({\mathbb{R}}^{d}):Ff(\xi )=0,\xi \notin
\Omega \}.
\end{eqnarray*}%
Let $D_{T}$ be the orthogonal projection of $L^{2}({\mathbb{R}}^{d})$ onto $%
\mathcal{D}(T)$, given explicitly by the multiplication of a characteristic
function of the set $T$ by $f$: 
\begin{equation*}
D_{T}f(t)=\chi _{T}(t)f(t)
\end{equation*}%
and let $B_{\Omega }$ be the orthogonal projection of $L^{2}({\mathbb{R}}%
^{d})$ onto $\mathcal{B}(\Omega )$, given explicitly as%
\begin{equation*}
B_{\Omega }f=F^{-1}\chi _{\Omega }Ff=\frac{1}{(2\pi )^{d/2}}\int_{{\mathbb{R}%
}^{d}}h(x-y)f(y)dy,
\end{equation*}%
where $Fh=\chi _{\Omega }$.\ The following Theorem, comprising Lemma 1 and
Theorem 1 of \cite{Landau1975} gives important information concerning the
spectral problem associated to the operator $D_{rT}B_{\Omega }D_{rT}$. This
information will be essential in our proofs. The notation $o(r^{d})$ refers
to behavior as $r\rightarrow \infty $.

\textbf{Theorem A }\cite{Landau1975}. \emph{The operator }$D_{rT}B_{\Omega
}D_{rT}$\emph{\ is bounded by }$1$\emph{, self-adjoint, positive, and
completely continuous. Denoting its set of eigenvalues, arranged in
nonincreasing order, by }$\{\lambda _{k}(r,T,\Omega )\}$\emph{, we have}

\begin{eqnarray*}
\sum_{k=0}^{\infty }\lambda _{k}(r,T,\Omega ) &=&r^{d}\left( 2\pi \right)
^{-d}\left\vert T\right\vert \left\vert \Omega \right\vert \\
\sum_{k=0}^{\infty }\lambda _{k}^{2}(r,T,\Omega ) &=&r^{d}\left( 2\pi
\right) ^{-d}\left\vert T\right\vert \left\vert \Omega \right\vert -o(r^{d})%
\text{.}
\end{eqnarray*}%
\emph{Moreover, given }$0<\gamma <1$\emph{, the number }$M_{r}(\gamma )$%
\emph{\ of eigenvalues which are not smaller than }$\gamma $\emph{,
satisfies, as }$r\rightarrow \infty $\emph{, }%
\begin{equation*}
M_{r}(\gamma )=\left( 2\pi \right) ^{-d}\left\vert T\right\vert \left\vert
\Omega \right\vert r^{d}+o(r^{d})\text{.}
\end{equation*}

We are now in a position to state and prove our main theorem. The lower
inequality is proved by constructing a set of orthonormal functions of $%
L^{2}({\mathbb{R}}^{d})$ satisfying (\ref{localization}). The proof of the
upper inequality uses some of the techniques contained in Landau's proof of
the non-hermitian Szeg\"{o}-type theorem \cite[Theorem 3]{Landau1975}.

\begin{theorem}
Let $\eta (\epsilon ,rT,\Omega )$ stand for the maximum number of
orthonormal functions $f\in L^{2}({\mathbb{R}}^{d})$ such that 
\begin{equation}
\left\Vert P_{rT,\Omega }f-f\right\Vert ^{2}\leq \epsilon .
\label{almostbandlimited}
\end{equation}%
Then, as $r\rightarrow \infty $, the following inequalities hold:%
\begin{equation}
\frac{\left\vert T\right\vert \left\vert \Omega \right\vert }{(2\pi )^{d}}%
(1+\epsilon )\leq \lim_{r\rightarrow \infty }\frac{\eta (\epsilon ,rT,\Omega
)}{r^{d}}\leq \frac{\left\vert T\right\vert \left\vert \Omega \right\vert }{%
(2\pi )^{d}}\left( 1-2\epsilon \right) ^{-1}\text{.}  \label{main}
\end{equation}
\end{theorem}

\begin{proof}
We first prove the lower inequality in (\ref{main}). Suppose (\ref%
{almostbandlimited}) holds for a positive real $\varepsilon $. Let $\sigma >0
$ be such that $\sigma ^{2}\leq \varepsilon $ and let $\mathcal{F}=\{\phi
_{k}\}$ be the normalized system of eigenfunctions of the operator $%
P_{rT,\Omega }$ with eigenvalues $\lambda _{k}>1-\sigma $. Now, given $f\in
L^{2}({\mathbb{R}}^{d})$, write%
\begin{equation}
f=\sum a_{k}\phi _{k}+h\text{,}  \label{expansf}
\end{equation}%
with $h\in $Ker$\left( P_{rT,\Omega }\right) $. Then%
\begin{equation}
P_{rT,\Omega }f=\sum a_{k}\lambda _{k}\phi _{k}  \label{fkern}
\end{equation}%
and%
\begin{eqnarray}
\left\Vert P_{rT,\Omega }f-f\right\Vert ^{2} &=&\left\Vert \sum (1-\lambda
_{k})a_{k}\phi _{k}+h\right\Vert ^{2}  \notag \\
&\leq &\sigma ^{2}\sum \left\vert a_{k}\right\vert ^{2}+\left\Vert
h\right\Vert ^{2}  \notag \\
&=&\sigma ^{2}\left\Vert f\right\Vert ^{2}+(1-\sigma ^{2})\left\Vert
h\right\Vert ^{2}\text{.}  \label{third}
\end{eqnarray}%
For the given $\sigma >0$ we pick a real number $\gamma $ such that 
\begin{equation}
\sigma ^{2}+(1-\sigma ^{2})\gamma =\varepsilon \text{,}  \label{epsilon}
\end{equation}%
Writing (\ref{epsilon}) as 
\begin{equation}
\gamma =\frac{\varepsilon -\sigma ^{2}}{1-\sigma ^{2}}\text{,}
\end{equation}%
it is clear that $\gamma $ is a positive increasing function of $\sigma $,
and that $\gamma \rightarrow \varepsilon $ as $\sigma \rightarrow 0$. Now
choose an integer number $n$ such that 
\begin{equation}
n\leq \frac{1}{\gamma }\leq n+1\text{.}  \label{integer}
\end{equation}%
We proceed further by considering the following partition of $\mathcal{F}$
into subsets $\mathcal{F}_{i}$, each of them containing $n$ functions:%
\begin{equation}
\mathcal{F}=\mathcal{F}_{1}\cup ...\cup \mathcal{F}_{l}\cup \mathcal{F}%
_{residual}\text{,}  \label{particion}
\end{equation}%
where the partition is made in such a way that the set $\mathcal{F}%
_{residual}$ contains only\ $o(r^{d})$ functions.\ This is possible to do
because Theorem A tells us that $\#\mathcal{F}=\frac{\left\vert T\right\vert
\left\vert \Omega \right\vert }{(2\pi )^{d}}r^{d}+o(r^{d})$.\ With each set $%
\mathcal{F}_{i}$ associate $h_{i}$ such that $h_{i}\in $Ker$\left(
P_{rT,\Omega }\right) $ and such that 
\begin{equation}
\left\langle h_{i},h_{j}\right\rangle =\delta _{i,j}\text{.}  \label{horthog}
\end{equation}%
This can be done since Ker$\left( P_{rT,\Omega }\right) $ has infinite
dimension, due to the inclusion $\mathcal{D}({\mathbb{R}}^{d}-rT)\subset $Ker%
$\left( P_{rT,\Omega }\right) $. Now, for each $i$, let $\{\psi
_{j}^{(i)}\}_{j=1}^{n+1}$ be a set of linear combinations of functions of $%
\mathcal{F}$ such that%
\begin{equation}
\left\langle \psi _{k}^{(i)},\psi _{j}^{(i)}\right\rangle =\left\{ 
\begin{array}{c}
-\frac{1}{n+1}\text{ \ \ \ \ }if\text{ \ \ }k\neq j \\ 
1-\frac{1}{n+1}\text{\ }if\text{ \ \ }k=j%
\end{array}%
\right. \text{,}  \label{ortfunct}
\end{equation}%
which can be constructed using a linear algebra argument as in the next
paragraph.

Consider a linear transformation $U:\mathbb{R}^{n}\longrightarrow \mathcal{F}%
_{i}$ mapping each vector of the canonical basis of $\mathbb{R}^{n}$ to each
of the given $n$ orthogonal functions of $\mathcal{F}_{i}$. Let $V$ be the
subspace of $\mathbb{R}^{n+1}$ which is orthogonal to the vector $v_{0}=%
\left[ \sqrt{\frac{1}{n+1}},...,\sqrt{\frac{1}{n+1}}\right] ^{T}\in \mathbb{R%
}^{n+1}$ and let $\{v_{1},...,v_{n}\}$ be an orthonormal basis of $V$.
Clearly, $\left\Vert v_{0}\right\Vert =1$ and, for $i=1,...,n$, $%
\left\langle v_{0},v_{i}\right\rangle =0$. Thus, the matrix%
\begin{equation*}
Q=\left[ 
\begin{array}{cccc}
v_{0} & v_{1} & ... & v_{n+1}%
\end{array}%
\right] \in \mathbb{R}^{(n+1)\times (n+1)}
\end{equation*}%
is orthogonal. If $u_{1},...,u_{n+1}\in \mathbb{R}^{n+1}$ are the rows of $Q$
then%
\begin{equation*}
Q^{T}=\left[ 
\begin{array}{cccc}
u_{1} &  & ... & u_{n+1}%
\end{array}%
\right] \in \mathbb{R}^{(n+1)\times (n+1)}
\end{equation*}%
is also orthogonal, we have $\left\langle u_{i},u_{j}\right\rangle =\delta
_{i,j}$. Let $u_{1}^{\prime },...,u_{n+1}^{\prime }\in \mathbb{R}^{n}$ be
the rows of $Q$ without the elements of the first column. They satisfy 
\begin{equation*}
\left\langle u_{k}^{\prime },u_{j}^{\prime }\right\rangle =\left\langle
u_{k},u_{j}\right\rangle -\frac{1}{n+1}=\left\{ 
\begin{array}{c}
-\frac{1}{n+1}\text{ \ \ \ \ }if\text{ \ \ }k\neq j \\ 
1-\frac{1}{n+1}\text{\ }if\text{ \ \ }k=j%
\end{array}%
\right. \text{,}
\end{equation*}%
and the functions in (\ref{ortfunct}) are obtained setting $\psi
_{j}^{(i)}=Uu_{j}^{\prime }$.

We are now in a position to construct the desired orthonormal system. Define
a sequence of orthonormal functions $\{\Phi _{j}^{(i)}\}_{i=1}^{l}$ using
the functions $\psi _{j}^{(i)}$ from (\ref{ortfunct}): 
\begin{equation}
\Phi _{j}^{(i)}=\psi _{j}^{(i)}+\sqrt{\frac{1}{n+1}}h_{i}\text{.}
\label{ortprinc}
\end{equation}%
Since $\psi _{j}^{(i)}$ are linear combinations of elements of $\mathcal{F}%
=\{\phi _{k}\}$, (\ref{ortprinc}) is a representation of the form (\ref%
{expansf}). Thus, (\ref{ortfunct}) and (\ref{horthog}) show that indeed $%
\left\langle \Phi _{k}^{(i)},\Phi _{j}^{(i)}\right\rangle $ $=\delta _{k,j}$
and we can apply (\ref{third}), (\ref{integer}) and (\ref{epsilon}) to obtain%
\begin{eqnarray*}
\left\Vert P_{rT,\Omega }\Phi _{j}^{(i)}-\Phi _{j}^{(i)}\right\Vert ^{2}
&\leq &\sigma ^{2}\left\Vert \Phi _{j}^{(i)}\right\Vert ^{2}+(1-\sigma
^{2})\left\Vert \sqrt{\frac{1}{n+1}}h_{i}\right\Vert ^{2} \\
&\leq &\sigma ^{2}+(1-\sigma ^{2})\gamma \\
&=&\varepsilon \text{.}
\end{eqnarray*}%
Thus, the functions in $\{\Phi _{j}^{(i)}\}_{j=1}^{n+1}$ verify (\ref%
{almostbandlimited}) and $\#\{\Phi _{j}^{(i)}\}_{j=1}^{n+1}=n+1$. We have
also $\#\mathcal{F}_{i}=n$, thus, 
\begin{equation*}
\#\{\Phi _{j}^{(i)}\}_{j=1}^{n+1}=\frac{n+1}{n}\#\mathcal{F}_{i}\text{.}
\end{equation*}%
Now, the cardinality of the union of all the sequences $\{\Phi _{j}^{(i)}\}$
obtained according to the above procedure is%
\begin{eqnarray*}
\#\left[ \cup _{i=1}^{l}\{\Phi _{j}^{(i)}\}_{j=1}^{n+1}\right] &=&\frac{n+1}{%
n}\#\left[ \cup _{i=1}^{l}\mathcal{F}_{i}\right] \\
&=&\frac{n+1}{n}\#\left[ \mathcal{F}-\mathcal{F}_{residual}\right] \\
&=&\frac{n+1}{n}(r^{d}\left( 2\pi \right) ^{-d}\left\vert T\right\vert
\left\vert \Omega \right\vert +o(r^{d})) \\
&\geq &\frac{\frac{1}{\gamma }+1}{\frac{1}{\gamma }}r^{d}\left( 2\pi \right)
^{-d}\left\vert T\right\vert \left\vert \Omega \right\vert +o(r^{d}) \\
&=&(1+\gamma )r^{d}\left( 2\pi \right) ^{-d}\left\vert T\right\vert
\left\vert \Omega \right\vert +o(r^{d})\text{.}
\end{eqnarray*}%
We have used Proposition 1 in the third equality (the fact that the
dimension of $\mathcal{F}$ is $r^{d}\left( 2\pi \right) ^{-d}\left\vert
T\right\vert \left\vert \Omega \right\vert +o(r^{d})$ and the fact that $%
\mathcal{F}_{residual}$ contains only $o(r^{d})$ functions). Denote by $%
M(rT,\Omega ,\epsilon )$ the minimum number of orthonormal functions
satisfying (\ref{almostbandlimited}). By construction we have obtained%
\begin{equation*}
M(rT,\Omega ,\epsilon )\geq \#\left[ \cup _{i=1}^{l}\{\Phi
_{j}^{(i)}\}_{j=1}^{n+1}\right] \geq (1+\gamma )r^{d}\left( 2\pi \right)
^{-d}\left\vert T\right\vert \left\vert \Omega \right\vert +o(r^{d})\text{.}
\end{equation*}%
and now we take $\sigma \rightarrow 0$, so that $\gamma \rightarrow \epsilon 
$ and we obtain%
\begin{equation*}
M(rT,\Omega ,\epsilon )\geq \#\left[ \cup _{i=1}^{l}\{\Phi
_{j}^{(i)}\}_{j=1}^{n+1}\right] \geq (1+\epsilon )r^{d}\left( 2\pi \right)
^{-d}\left\vert T\right\vert \left\vert \Omega \right\vert +o(r^{d})\text{.}
\end{equation*}%
This proves the lower inequality in (\ref{main}).

Let us now prove the upper inequality in (\ref{main}). Consider again $%
f=\sum a_{k}\phi _{k}+h$ with $h\in $Ker$\left( P_{rT,\Omega }\right) $.
Then, using (\ref{fkern}) and%
\begin{equation*}
\left\Vert B_{\Omega }D_{rT}f\right\Vert ^{2}=\left\langle P_{rT,\Omega
}f,f\right\rangle =\sum \left\vert a_{k}\right\vert ^{2}\lambda _{k}\text{,}
\end{equation*}%
together with the fact that $D_{rT}$ is a projection, one can write%
\begin{equation}
\left\Vert B_{\Omega }D_{rT}f-P_{rT,\Omega }f\right\Vert ^{2}=\left\Vert
B_{\Omega }D_{rT}f\right\Vert ^{2}-\left\Vert P_{rT,\Omega }f\right\Vert
^{2}=\sum \left\vert a_{k}\right\vert ^{2}\lambda _{k}(1-\lambda _{k})\text{.%
}  \label{triangle}
\end{equation}%
Now, for $\delta >0$ define $\mathcal{E}(\delta )$ as the subspace generated
by the eigenfunctions of $P_{rT,\Omega }$ such that the corresponding
eigenvalues satisfy $\delta <\lambda _{k}<1-\delta $ and let 
\begin{equation*}
\mathcal{F}(\delta )=\left\{ f\in L^{2}({\mathbb{R}}^{d}):\left\Vert
f\right\Vert =1\text{ \ \ }\sum_{\delta <\lambda _{k}<1-\delta }\left\vert
a_{k}\right\vert ^{2}\leq \delta \right\} \text{.}
\end{equation*}%
For $f\in \mathcal{F}(\delta )$,%
\begin{eqnarray*}
&&\left\Vert B_{\Omega }D_{rT}f-P_{rT,\Omega }f\right\Vert ^{2} \\
&=&\sum_{\lambda _{k}\leq \delta }\left\vert a_{k}\right\vert ^{2}\lambda
_{k}(1-\lambda _{k})+\sum_{\delta <\lambda _{k}<1-\delta }\left\vert
a_{k}\right\vert ^{2}\lambda _{k}(1-\lambda _{k})+\sum_{\lambda _{k}\geq
1-\delta }\left\vert a_{k}\right\vert ^{2}\lambda _{k}(1-\lambda _{k})\leq
3\delta \text{.}
\end{eqnarray*}%
Thus, $\delta $ can be chosen in such a way that 
\begin{equation}
\left\Vert B_{\Omega }D_{rT}f-P_{rT,\Omega }f\right\Vert ^{2}\leq
\varepsilon \text{.}  \label{approximation}
\end{equation}%
Let us assume the existence of a set $\mathcal{N}$ of $\eta (\epsilon
,rT,\Omega )$ orthonormal functions of $L^{2}({\mathbb{R}}^{d})$ satisfying (%
\ref{almostbandlimited}). To estimate how many of them belong to $\mathcal{F}%
(\delta )$, consider two subspaces $\mathcal{E}$ and $\mathcal{G}$ with
corresponding projections $E,G,$ and dimensions $e$ and $g$ respectively,
with $e<g$. Let $v_{1},...,v_{g}$ be an orthonormal set in $\mathcal{G}$.
Then $\sum \left\Vert Ev_{i}\right\Vert ^{2}=\sum \left( Ev_{i},v_{i}\right)
=\sum \left( GEGv_{i},v_{i}\right) $ represents the trace of the operator $%
GEG$, independent of the choice of basis. Choose the basis $\{w_{i}\}$ such
that the first vectors are in $\mathcal{GE}$ and the remaining vectors in
the orthogonal complement in $\mathcal{G}$ of $\mathcal{GE}$ (the image of $%
GE$). For each of the latter, $\left( GEGw,w\right) =0$, while the dimension
of $\mathcal{GE}$ is at most $e$. Hence $\sum \left\Vert Ev_{i}\right\Vert
^{2}=\sum_{1}^{g}\left( Ew_{i},w_{i}\right) \leq \sum_{1}^{e}\left(
GEGw_{i},w_{i}\right) \leq e$. Thus, the number of orthonormal vectors $%
\{v_{i}\}$ for which $\left\Vert Ev_{i}\right\Vert ^{2}\geq \delta $ cannot
exceed $e/\delta $.

As a result of the previous paragraph, after excluding from $\mathcal{N}$ at
most $\delta ^{-1}\dim \mathcal{E}(\delta )$ elements, those remaining are
in $\mathcal{F}(\delta )$. Since, from Theorem A, we have $\dim \mathcal{E}%
(\delta )=o(r^{d})$, there are $\eta (\epsilon ,rT,\Omega )-o(r^{d})$
functions in $\mathcal{N}\cap \mathcal{F}(\delta )$. Let $f$ be one of them.
Now we can use (\ref{almostbandlimited}), (\ref{approximation}) and the
triangle inequality to obtain%
\begin{equation*}
1-\left\Vert B_{\Omega }D_{rT}f\right\Vert ^{2}\leq \left\Vert B_{\Omega
}D_{rT}f-f\right\Vert \leq 2\varepsilon \text{,}
\end{equation*}%
leading to $\left\Vert B_{\Omega }D_{rT}f\right\Vert ^{2}\geq 1-2\varepsilon 
$, for each of the $\eta (\epsilon ,rT,\Omega )-o(r^{d})$ orthonormal
functions. Since $\left\Vert B_{\Omega }D_{rT}f\right\Vert ^{2}=\left\langle
P_{rT,\Omega }f,f\right\rangle $, the sum of these terms for any orthonormal
set cannot exceed the trace of $D_{rT}B_{\Omega }D_{rT}$. Thus, using the
trace from Theorem A, we conclude that%
\begin{equation*}
(1-2\varepsilon )\left( \eta (\epsilon ,rT,\Omega )-o(r^{d})\right) \leq
\sum_{k=0}^{\infty }\lambda _{k}(r,T,\Omega )=r^{d}\left( 2\pi \right)
^{-d}\left\vert T\right\vert \left\vert \Omega \right\vert \text{,}
\end{equation*}%
leading to the upper inequality in (\ref{main}).
\end{proof}

\begin{remark}
In the case where $T$ and $\Omega $ are finite unions of bounded intervals,
the term $o(r)$ in Theorem A can be replaced by $\log r$ \cite{LP}, \cite%
{Landau1967}. Thus, (\ref{estimateintrol2}) follows using this estimate in
our proofs of Theorem 1. See the recent monograph \cite{HogLak} for more
estimates on the eigenvalues of the time- and band- limiting operator.
\end{remark}

\begin{remark}
It is possible to obtain an analogue of Theorem 1 in the set up of the
Hankel transform. The result corresponding to Theorem A has been proved in 
\cite{AB}.
\end{remark}

\begin{remark}
The proof of the lower inequality\ in (\ref{main}) constructs a new set of
orthogonal functions. On the one side we don't know yet to what extent such
functions can be used in applications. On the other side the lower
inequality in (\ref{main}) may provide useful information in cases where
signals are approximated by functions which are not optimal concentrated as
the prolates, but still have some concentration properties. This is the case
of the Hermite functions, where an estimate of the energy left outside $%
\Omega $ may provide an indication of the increase in the number of
functions required to avoid undersampling.
\end{remark}

\section{Gabor localization operators}

The Gabor (or short-time Fourier) transform of a function or distribution $f$
with respect to a window function $g\in L^{2}(\mathbb{R}^{d})$ is defined to
be, for $z=(x,\xi )\in \mathbb{R}^{2d}$: 
\begin{equation}
\mathcal{V}_{g}f(z)=\mathcal{V}_{g}f(x,\xi )=\int_{\mathbb{R}^{d}}f(t)%
\overline{g(t-x)}e^{-2\pi i\xi t}dt\text{.}  \label{Gabor}
\end{equation}%
The following relations are usually called the orthogonal relations for the
short-time Fourier transform. Let $f_{1},f_{2},g_{1},g_{2}\in L^{2}(%
\mathbb{R}
^{d})$. Then $V_{g_{1}}f_{1},V_{g_{2}}f_{2}\in L^{2}(%
\mathbb{R}
^{2d})$ and 
\begin{equation}
\int \int_{\mathbb{R}^{2d}}V_{g_{1}}f_{1}(x,\xi )\overline{%
V_{g_{2}}f_{2}(x,\xi )}dxd\xi =\left\langle f_{1},f_{2}\right\rangle _{L^{2}(%
\mathbb{R}
^{d})}\overline{\left\langle g_{1},g_{2}\right\rangle }_{L^{2}(%
\mathbb{R}
^{d})}\text{.}  \label{ortogonalityrelations}
\end{equation}%
The localization operator which concentrates the time-frequency content of a
function in the region $S$ operator $\mathcal{C}_{S}:L^{2}(%
\mathbb{R}
^{d})\rightarrow L^{2}(%
\mathbb{R}
^{d})$ can be defined weakly as%
\begin{equation*}
\left\langle \mathcal{C}_{S}f,h\right\rangle =\int \int_{S}\mathcal{V}%
_{g}f(x,\xi )\overline{\mathcal{V}_{g}h(x,\xi )}dxd\xi \text{,}
\end{equation*}%
for all $f,g\in L^{2}(%
\mathbb{R}
^{d}).$ These operators have been introduced in time-frequency analysis by
Daubechies \cite{da88}. Since then, applications and connections to several
mathematical topics, namely complex and harmonic analysis \cite{Seip0}, \cite%
{AD}, \cite{cogr03}, \cite{GT} have been found. The eigenvalue problem has
been object of a detailed study in \cite{RT}, \cite{FN} and \cite{FeiNow}.

The image of $L^{2}(%
\mathbb{R}
^{d})$ under the Gabor transform with the window $g$ will be named as the 
\emph{Gabor space} $\mathcal{G}_{g}$. It is the following subspace of $L^{2}(%
\mathbb{R}
^{2d})$:%
\begin{equation*}
\mathcal{G}_{g}=\left\{ V_{g}f:f\in L^{2}(%
\mathbb{R}
^{d})\right\} \text{.}
\end{equation*}%
The reproducing kernel of the Gabor space $\mathcal{G}_{g}$ is 
\begin{equation}
K_{g}(z,w)=\left\langle \pi _{z}g,\pi _{w}g\right\rangle _{L^{2}(%
\mathbb{R}
^{d})}  \label{repdef}
\end{equation}%
and the projection operator $\mathcal{P}_{g}:L^{2}(%
\mathbb{R}
^{2d})\rightarrow \mathcal{G}_{g}$,%
\begin{equation*}
\mathcal{P}_{g}F(z)=\int F(w)\overline{K_{g}(z,w)}dw\text{.}
\end{equation*}%
It is shown in \cite{RT} that, for $F\in \mathcal{G}_{g},$%
\begin{equation*}
\mathcal{V}_{g}\mathcal{C}_{S}\mathcal{V}_{g}^{-1}F(z)=\int_{S}F(w)\overline{%
K_{g}(z,w)}dw=\mathcal{P}_{g}D_{S}F(z)\text{.}
\end{equation*}%
For the whole $L^{2}({\mathbb{R}}^{2d})$ one can write%
\begin{equation*}
\mathcal{V}_{g}\mathcal{C}_{S}\mathcal{V}_{g}^{\ast }=\mathcal{P}_{g}D_{S}%
\text{.}
\end{equation*}%
Thus, the spectral properties of $\mathcal{C}_{S}$\ are identical to those
of $\mathcal{P}_{g}D_{S}$. Moreover, the\ operator $D_{S}\mathcal{P}%
_{g}D_{S} $ in $L^{2}({\mathbb{R}}^{2d})$ and the operator $\mathcal{P}%
_{g}D_{S}$ have the same nonzero eigenvalues with multiplicity (see Lemma 1
in \cite{RT}). The analogue of Theorem A in this context is the following.

\textbf{Theorem B }\cite{RT}.\emph{\ The operator }$D_{rS}P_{g}D_{rS}$\emph{%
\ is bounded by }$1$\emph{, self-adjoint, positive, and completely
continuous. Denoting its set of eigenvalues, arranged in nonincreasing
order, by }$\{\lambda _{k}(rS)\}$\emph{, we have}%
\begin{eqnarray*}
\sum_{k=0}^{\infty }\lambda _{k}(rS) &=&r^{d}\left\vert S\right\vert \\
\sum_{k=0}^{\infty }\lambda _{k}^{2}(rS) &=&r^{d}\left\vert S\right\vert
-o(r^{d})\text{.}
\end{eqnarray*}%
\emph{Moreover, given }$0<\gamma <1$\emph{, the number }$M_{r}(\gamma )$%
\emph{\ of eigenvalues which are not smaller than }$\gamma $\emph{,
satisfies, as }$r\rightarrow \infty $\emph{, }%
\begin{equation*}
M_{r}(\gamma )=r^{d}\left\vert S\right\vert +o(r^{d})\text{.}
\end{equation*}

Now that we have described the Gabor set-up in a close analogy to the band-
time- limiting case, we obtain an analogue of Theorem 1 by performing minor
adaptations in the proof.

\begin{theorem}
Let $\eta (\epsilon ,rS)$ stand for the maximum number of orthogonal
functions $F\in L^{2}({\mathbb{R}}^{2d})$ such that 
\begin{equation}
\left\Vert D_{rS}\mathcal{P}_{g}D_{rS}F-F\right\Vert ^{2}\leq \epsilon \text{%
.}  \label{almostbandlimitedTF}
\end{equation}%
Then, as $r\rightarrow \infty $, the following inequalities hold:%
\begin{equation*}
\left\vert S\right\vert (1+\epsilon )\leq \lim_{r\rightarrow \infty }\frac{%
\eta (\epsilon ,rS)}{r^{2d}}\leq \frac{\left\vert S\right\vert }{1-2\epsilon 
}\text{.}
\end{equation*}
\end{theorem}

\begin{proof}
The proof mimics the proof of Theorem 1, replacing $D_{rT}B_{\Omega }D_{rT}$
by $D_{rS}\mathcal{P}_{g}D_{rS}$, $B_{\Omega }D_{rT}$ by $\mathcal{P}%
_{g}D_{rS}$ and Theorem A by Theorem B.
\end{proof}

\textbf{Acknowledgement. }The authors thank Jos\'{e} Luis Romero for his
constructive criticism of earlier versions of the manuscript, leading to a
better formulation of the results. We also want to thank both reviewers for
the very careful reading of the manuscript, leading to several corrections,
improvements on the readability and also valuable insights into the
mathematical content.

\end{document}